\documentclass[11pt,reqno]{amsart}
\usepackage{a4wide}

\usepackage[utf8]{inputenc}
\usepackage[T1]{fontenc}
\usepackage{amsmath,amsthm}
\usepackage{amsfonts,amssymb,enumerate}
\usepackage{url,paralist}
\usepackage{listings}

\usepackage{mathtools,color}
\usepackage[colorlinks=true,urlcolor=blue,linkcolor=red,citecolor=magenta,hypertexnames=false]{hyperref}

\usepackage{enumerate}
\usepackage{enumitem}
\setlist[enumerate,1]{label={(\roman*)}}
\usepackage[arrow,curve,matrix,tips,2cell]{xy}
\SelectTips{eu}{10} \UseTips
\UseAllTwocells
\usepackage{tikz}
\tikzset{auto}
\usetikzlibrary{cd}
\usepackage{pgfplotstable}
\usepackage{pgfplots}
\usepackage{lscape}
\usepackage{cleveref}
\Crefname{subsection}{subsection}{subsections}
\Crefname{subsection}{Subsection}{Subsections}

\usepackage[colorinlistoftodos]{todonotes}
\definecolor{mygreen}{rgb}{0,0.6,0}
\definecolor{mygray}{rgb}{0.5,0.5,0.5}
\definecolor{mymauve}{rgb}{0.58,0,0.82}

\usepackage{inconsolata}
\usepackage{adjustbox}
\usepackage{array}

\newcolumntype{R}[2]{%
    >{\adjustbox{angle=#1,lap=\width-(#2)}\bgroup}%
    l%
    <{\egroup}%
}

\theoremstyle{plain}
\newtheorem{theorem}{Theorem}[section]
\newtheorem{lemma}[theorem]{Lemma}

\newtheorem{corollary}[theorem]{Corollary}

\newtheorem{proposition}[theorem]{Proposition}
\newtheorem{conjecture}[theorem]{Conjecture}
\newtheorem*{theorem*}{Theorem}

\newtheorem*{claim*}{Claim}
\newtheorem*{lemma*}{Lemma}

\theoremstyle{definition}
\newtheorem{definition}[theorem]{Definition}

\newtheorem{example}[theorem]{Example}
\newtheorem{examples}[theorem]{Examples}
\newtheorem{remark}[theorem]{Remark}

\newcommand\RP{\mathbb{RP}}
\newcommand{\R}{\mathbb{R}}
\newcommand{\K}{\mathcal{K}}

\newcommand{\Z}{\mathbb{Z}}

\newcommand{\A}{\mathcal{A}}
\newcommand{\C}{\mathcal{C}}
\newcommand{\Hi}{\mathcal{H}}
\newcommand{\B}{\mathcal{B}}

\renewcommand{\P}{\mathcal{P}}
\newcommand{\Pop}{\P^{\operatorname{op}}}
\renewcommand{\O}{\mathcal{O}}

\newcommand{\At}{\operatorname{Atoms}}
\newcommand{\coAt}{\operatorname{coAtoms}}

\newcommand{\op}{{\operatorname{op}}}

\renewcommand{\int}{\operatorname{int}}

\newcommand{\codestyle}[1]{{\tt #1}}
\newcommand{\polymake}{\texttt{polymake}}
\newcommand{\sage}{\texttt{SageMath}}
\newcommand{\sageversion}{\texttt{sage-8.9}}
\newcommand{\normaliz}{\texttt{normaliz}}

\begin{document}
\title[A new face iterator for finite locally branched lattices]{A new face iterator for polyhedra and for\\ more general finite locally branched lattices}

\author[J.~Kliem]{Jonathan Kliem}
\address[J.~Kliem]{Institut f\" ur Mathematik, Freie Universität Berlin, Germany}
\email{jonathan.kliem@fu-berlin.de}
\author[C.~Stump]{Christian Stump}
\address[C.~Stump]{Fakultät für Mathematik, Ruhr-Universität Bochum, Germany}
\email{christian.stump@rub.de}

\thanks{
  J.K.~receives funding by the Deutsche Forschungsgemeinschaft DFG under Germany´s Excellence Strategy – The Berlin Mathematics Research Center MATH+ (EXC-2046/1, project ID: 390685689).
  C.S.~is supported by the DFG Heisenberg grant STU 563/4-1 ``Noncrossing phenomena in Algebra and Geometry''.
}

\begin{abstract}
  We discuss a new memory-efficient depth-first algorithm and its implementation that iterates over all elements of a finite locally branched lattice.
  This algorithm can be applied to face lattices of polyhedra and to various generalizations such as finite polyhedral complexes and subdivisions of manifolds, extended tight spans and closed sets of matroids.
  Its practical implementation is very fast compared to state-of-the-art implementations of previously considered algorithms.
  Based on recent work of Bruns, García-Sánchez, O'Neill and Wilburne, we apply this algorithm to prove \emph{Wilf's conjecture} for all numerical semigroups of multiplicity~$19$ by iterating through the faces of the \emph{Kunz cone} and identifying the possible \emph{bad faces} and then checking that these do not yield counterexamples to Wilf's conjecture.
\end{abstract}

\maketitle

\newcommand{\defn}[1]{\emph{#1}} %

\section{Introduction}

We call a finite lattice $(\P, \leq)$ \defn{locally branched} if all intervals of length two contain at least four elements.
We show that such lattices are atomic and coatomic and refer to \Cref{sec:framework} for details.

\medskip

This paper describes a depth-first algorithm to iterate through the elements in a finite locally branched lattice given its coatoms, see \Cref{sec:algorithm}.
It moreover describes variants of this algorithm allowing the iteration over slightly more general posets.
Most importantly, examples of such locally branched lattices (or its mild generalizations) include face posets of
\begin{itemize}
\setlength{\itemsep}{3pt}
  \item polytopes and unbounded polyhedra,
  \item finite polytopal or polyhedral complexes,
  \item finite polyhedral subdivisions of manifolds,
  \item extended tight spans, and
  \item closed sets of matroids.
\end{itemize}
One may in addition compute all cover relations as discussed in \Cref{sec:covers}.
The provided theo\-re\-ti\-cal runtime (without variants) is the same as of the algorithm discussed by V.~Kaibel and M.~E.~Pfetsch in~\cite{Kaibel2002}, see \Cref{sec:data_structures}.
In the sligthly generalized situations, the theoretical runtime might be better as for extended tight spans (without chords) with many facets (using the opposite lattice), or might be worse as for extended tight spans with many vertices.

\medskip

In practice it appears that the chosen data structures and implementation details make the implementation\footnote{See~\url{https://trac.sagemath.org/ticket/26887}, merged into \sage\ version \sageversion.} very fast for the iteration and still fast for cover relations in the graded case compared to state-of-the-art implementations of previously considered algorithms, see \Cref{sec:performance}.

In \Cref{sec:wilf}, we apply the presented algorithm to affirmatively settle \emph{Wilf's conjecture} for all numerical semigroups of multiplicity~$19$ by iterating, up to a certain symmetry of order~$18$, through all faces of the \emph{Kunz cone} (which is a certain unbounded polyhedron), identifing the \emph{bad faces} which possibly yield counterexamples to Wilf's conjecture, and then checking that these do indeed not yield such counterexamples.
This is based on recent work of W.~Bruns, P.~García-Sánchez, C.~O'Neill and D.~Wilburne~\cite{Bruns} who developed this approach to the conjecture and were able to settle it up to multiplicity~$18$.

\medskip

In \Cref{sec:runtimes}, we finally collect detailed runtime comparisions between the implementation of the presented algorithm with the state-of-the-art implementations in \polymake\ and in \normaliz.

\subsection*{Acknowledgements}

We thank Michael Joswig and Winfried Bruns for valuable discussions and for providing multiple relevant references.
We further thank Jean-Philippe Labbé for pointing us to \cite{Bruns} and all participants of the trac ticket in \sage\footnotemark[1] for stimulating discussions.

\section{Formal framework}
\label{sec:framework}

Let $(\P, \leq)$ be a finite poset and denote by $\prec$ its cover relations.
We usually write $\P$ for $(\P,\leq)$ and write $\Pop$ for the opposite poset $(\Pop,\leq_{\operatorname{op}})$ with $b \leq_{\operatorname{op}} a$ if $a \leq b$.
For $a,b \in \P$ with $a \leq b$ we denote the interval as $[a,b] = \{p \in \P \mid a \leq p \leq b\}$.
If~$\P$ has a lower bound~$\hat{0}$, its \defn{atoms} are the upper covers of the lower bound,
\[
  \At(\P) = \{p \in \P \mid \hat{0} \prec p\}
\]
and, for $p \in \P$, we write $\At(p) = \{a \in \At(\P) \mid p \geq a\}$ for the atoms below~$p$.
Analogously, if~$\P$ has an upper bound $\hat{1}$, its \defn{coatoms} are the lower covers of the upper bound, $\coAt(\P) = \{p \in \P \mid p \prec \hat{1}\}$.
$\P$ is called \defn{graded} if it admits a rank function $r \colon \P \to \Z$ with $p \prec q \Rightarrow r(p) + 1 = r(q)$.

\begin{definition}
  $\P$ is locally branched if for every saturated chain $a \prec b \prec c$ there exists an element~$d \neq b$ with $a < d < c$.
  If this element is unique, then~$\P$ is said to have the \defn{diamond property}.
\end{definition}

The diamond property is a well-known property of face lattices of polytopes, see~\cite[Theorem~2.7(iii)]{ZieglerLecturesPolytopes}.
The property of being locally branched has also appeared in the literature in contexts different from the present under the name \emph{$2$-thick lattices}, see for example~\cite{MR1935777} and the references therein.

\medskip

An obvious example of a locally branched lattice is the boolean lattice~$B_n$ given by all subsets of $\{1,\dots,n\}$ ordered by containment.
We will later see that all locally branched lattices are meet semi-sublattices of~$B_n$.

\medskip

In the following, we assume~$\P$ to be a finite lattice with meet operation~$\wedge$, join operation~$\vee$, lower bound~$\hat{0}$ and upper bound~$\hat{1}$.
We say that
\begin{itemize}
  \item~$\P$ is \defn{atomic} if all elements are joins of atoms,
  \item~$\P$ is \defn{coatomic} if all elements are meets of coatoms,
  \item $p \in \P$ is \defn{join-irreducible} if~$p$ has a unique lower cover $q \prec p$,
  \item $p \in \P$ is \defn{meet-irreducible} if~$p$ has a unique upper cover $p \prec q$.
\end{itemize}

Atoms are join-irreducible and coatoms are meet-irreducible.
The following classification of atomic and coatomic lattices is well-known.

\begin{lemma}
\label{lem:meet_irreducible}
  We have that
  \begin{enumerate}
      \item~$\P$ is atomic if and only if the only join-irreducible elements are the atoms,\label{item:atomic_join-irreducible}
      \item~$\P$ is coatomic if and only if the only meet-irreducible elements are the coatoms.
    \end{enumerate}
\end{lemma}
\begin{proof}
  First observe that for all $p,q \in \P$ we have $p \geq q \Rightarrow \At(p) \supseteq \At(q)$ and $p \geq \bigvee\At(p)$.
  Moreover,~$\P$ is atomic if and only if $p = \bigvee\At(p)$ for all $p \in \P$.

  Assume that $\P$ is atomic and let $q \in \P$ join-irreducible and $p \prec q$.
  Because we have $\At(p) \neq \At(q)$ it follows that $p = \hat 0$.
  Next assume that $\P$ is not atomic and let $p \in \P$ minimal such that $p > \bigvee \At(p)$.
  If $q < p$ then by minimality $q = \bigvee \At(q)$.
  It follows that $q \leq \bigvee \At(p)$ and $p$ is join-irreducible.

  The second equivalence is the first applied to~$\Pop$.
\end{proof}

\begin{examples}
  The face lattice of a polytope has the diamond property, it is atomic and coatomic, and every interval is again the face lattice of a polytope.
  The face lattice of an (unbounded) polyhedron might neither be atomic nor coatomic as witnessed by the face lattice of the positive orthant in $\R^2$ with five faces.
\end{examples}

The reason to introduce locally branched posets is the following relation to atomic and coatomic lattices, which has, to the best of our knowledge, not appeared in the literature.

\begin{proposition}
\label{prop:locally_branched_atomic_coatomic}
  The following statements are equivalent:
  \begin{enumerate}
      \item~$\P$ is locally branched,  \label{item:locally_branched}
      \item every interval of~$\P$ is atomic,  \label{item:atomic}
      \item every interval of~$\P$ is coatomic.  \label{item:coatomic}
  \end{enumerate}
\end{proposition}

\begin{proof}
  $\P$ is locally branched if and only if $\Pop$ is locally branched.
  Also,~$\P$ is atomic if and only if $\Pop$ is coatomic.
  Hence, it suffices to show~\ref{item:locally_branched} $\Leftrightarrow$~\ref{item:atomic}.
  Suppose~$\P$ is not locally branched.
  Then, there exist $p \prec x \prec q$ such that the interval~$[p,q]$ contains exactly those three elements.
  Clearly, $[p,q]$ is not atomic.
  Now suppose $[p,q] \subseteq \P$ is not atomic and~$x$ is join-irreducible with unique lower cover~$y$ with $p < y \prec x$.
  There exists $z \in [p,q]$ with $z \prec y$ and the interval $[z,x]$ contains exactly those three elements.
\end{proof}

\begin{example}
    On the left an example of a non-graded locally branched lattice.
    On the right an example of an atomic, coatomic lattice, which is not locally branched as the interval between the two larger red elements contains only three elements.

    
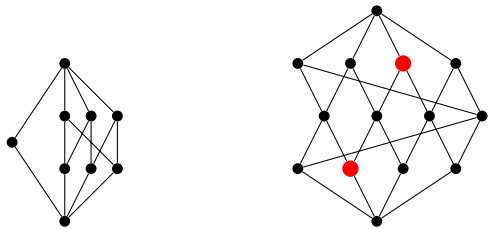
\begin{center}
        \begin{tikzpicture}[scale=.7]
            \node (b) at (0,3) [circle,fill=black,scale=0.4] {};
            \node (oneface) at (-1,1.5) [circle,fill=black,scale=0.4] {};

            \node (lower1) at (0,1) [circle,fill=black,scale=0.4] {};
            \node (lower2) at (0.5,1) [circle,fill=black,scale=0.4] {};
            \node (lower3) at (1,1) [circle,fill=black,scale=0.4] {};
            \node (upper1) at (0,2) [circle,fill=black,scale=0.4] {};
            \node (upper2) at (0.5,2) [circle,fill=black,scale=0.4] {};
            \node (upper3) at (1,2) [circle,fill=black,scale=0.4] {};

            \node (a) at (0,0) [circle,fill=black,scale=0.4] {};
            \draw (b) -- (oneface) -- (a);
            \draw (b) -- (upper1);
            \draw (b) -- (upper2);
            \draw (b) -- (upper3);
            \draw (a) -- (lower1);
            \draw (a) -- (lower2);
            \draw (a) -- (lower3);
            \draw (upper1) -- (lower1) -- (upper2) -- (lower2) -- (upper3) -- (lower3) -- (upper1);
        \end{tikzpicture}
        \hspace{2cm}
        \begin{tikzpicture}[scale=.7]
            \node (b) at (0,4) [circle,fill=black,scale=0.4] {};
            \node (a) at (0,0) [circle,fill=black,scale=0.4] {};

            \node (lower1) at (-1.5,1) [circle,fill=black,scale=0.4] {};
            \node (lower2) at (-0.5,1) [circle,fill=red,scale=0.6] {};
            \node (lower3) at (0.5,1) [circle,fill=black,scale=0.4] {};
            \node (lower4) at (1.5,1) [circle,fill=black,scale=0.4] {};
            \node (upper1) at (-1.5,3) [circle,fill=black,scale=0.4] {};
            \node (upper2) at (-0.5,3) [circle,fill=black,scale=0.4] {};
            \node (upper3) at (0.5,3) [circle,fill=red,scale=0.6] {};
            \node (upper4) at (1.5,3) [circle,fill=black,scale=0.4] {};
            \node (middle1) at (-1,2) [circle,fill=black,scale=0.4] {};
            \node (middle2) at (0,2) [circle,fill=black,scale=0.4] {};
            \node (middle3) at (1,2) [circle,fill=black,scale=0.4] {};
            \node (middle4) at (2,2) [circle,fill=black,scale=0.4] {};
            \draw (a) -- (lower1) -- (middle1) -- (upper1) -- (b);
            \draw (a) -- (lower2) -- (middle2) -- (upper2) -- (b);
            \draw (a) -- (lower3) -- (middle3) -- (upper3) -- (b);
            \draw (a) -- (lower4) -- (middle4) -- (upper4) -- (b);
            \draw (lower1) -- (middle4) -- (upper1);
            \draw (lower2) -- (middle1) -- (upper2);
            \draw (lower3) -- (middle2) -- (upper3);
            \draw (lower4.center) -- (middle3) -- (upper4);
        \end{tikzpicture}
    \end{center}
\end{example}

Let~$\P$ be a finite locally branched poset with atoms $\{1,\dots,n\}$.
As we have seen that~$\P$ is atomic and thus $p = \bigvee \At(p)$ for all $p \in \P$.
The following proposition underlines the importance of subset checks and of computing intersections to understanding finite locally branched lattices.

\begin{proposition}
\label{prop:atom_representation_of_elements}
  In a finite locally branched lattice it holds that
  \begin{enumerate}
      \item $p \leq q \Leftrightarrow \At(p) \subseteq \At(q).$  \label{item:subsetcheck}
      \item $p \wedge q = \bigvee \left(\At(p) \cap \At(q)\right).$  \label{item:meet}
  \end{enumerate}
\end{proposition}
\begin{proof}\
  \begin{enumerate}
    \item If $p \leq q$ then clearly $\At(p) \subseteq \At(q)$.
      On the other hand, if $\At(p) \subseteq \At(q)$, then $p = \bigvee \At(p) \leq \bigvee \At(q) = q$, as $\bigvee \At(q)$ is in particular an upper bound for $\At(p)$.
    \item By~\ref{item:subsetcheck} it holds that $\bigvee \left(\At(p) \cap \At(q)\right)$ is a lower bound of~$p$ and $q$.
      Also, $\At(p \wedge q) \subseteq \At(p), \At(q)$ and we obtain
      \[
        \At(p \wedge q) \subseteq \At(p) \cap \At(q). \qedhere
      \]
  \end{enumerate}
\end{proof}

This proposition provides the following meet semi-lattice embedding of any finite locally branched lattice into a boolean lattice.

\begin{corollary}
\label{cor:embedding_to_boolean}
  Let~$\P$ be a finite locally branched lattice with $\At(\P) = \{1, \dots ,n\}$.
  The map $p \mapsto \At(p)$ is a meet semi-sublattice embedding of~$\P$ into the boolean lattice~$B_n$.
\end{corollary}

\begin{example}
  The above embedding does not need to be a join semi-sublattice embedding as witnessed by the face lattice of a square in $\R^2$.
\end{example}

\begin{remark}
    \Cref{prop:atom_representation_of_elements} shows that checking whether the relation $p \leq q$ holds in~$\P$ is algorithmically a subset check $\At(p) \subseteq \At(q)$, while computing the meet is given by computing the intersection $\At(p) \cap \At(q)$.
\end{remark}

Justified by \Cref{cor:embedding_to_boolean}, we restrict our attention in this paper to meet semi-sublattices of the boolean lattice.

\subsection{Variants of this framework and examples}
\label{sec:frameworkvariants}

Before presenting in \Cref{sec:algorithm} the algorithm
to iterate over the elements of a finite locally branched lattice together with variants to avoid any element above certain atoms and to avoid any element below certain coatoms (or other elements of $B_n$),
we give the the following main use cases for such an iterator.

\begin{example}[Polytope]
\label{ex:polytope}
  The face lattice of a polytope~$P$ has the diamond property and is thus locally branched.
\end{example}

\begin{example}[Polyhedron]
\label{ex:polyhedron}
  The face lattice of a polyhedron~$P$ is isomorphic to the one obtained from quotiening out the affine space~$P$ contains.
  Thus, we can assume that~$P$ does not contain an affine line.
  It is well known (see e.g.~\cite[Exercise~2.19]{ZieglerLecturesPolytopes}) that we may add an extra facet $\overline{F}$ to obtain a polytope $\overline{P}$.
  The faces of~$P$ are exactly the faces of $\overline{P}$ not contained in $\overline{F}$ (together with the empty face).
  Thus, the iterator visits all non-empty faces of~$P$ by visiting all faces of $\overline{P}$ not contained in $\overline{F}$.
\end{example}

\begin{example}[Polytopal subdivision of manifold]
\label{ex:manifold}
  The face poset of a finite polytopal subdivision of a closed manifold (compact manifold without boundary).
  Adding an artificial upper bound~$\hat{1}$, this is a finite locally branched lattice.%
\end{example}

\begin{example}[Extended tight spans]
\label{ex:tight_spans}
  We consider extended tight spans as defined in \cite[Section~3]{Hampe2019} as follows:
  Let $P \subset \R^d$ be a finite point configuration, and let~$\Sigma$ be a polytopal complex with vertices~$P$, which covers the convex hull of~$P$.
  We call the maximal cells of~$\Sigma$ facets.
  We can embedd~$\Sigma$ into a closed $d$-manifold $M$:
  In any case, we can add a vertex at infinity and for each face $F$ on the boundary of~$\Sigma$ a face $F \cup \{\infty\}$.
  In many cases, just adding one facet containing all vertices on the boundary will work as well.

  Given a collection~$\Gamma$ of boundary faces of~$\Sigma$.
  We can iterate over all elements of~$\Sigma$, which have empty intersection with~$\Gamma$:
  Iterate over all faces of $M$, which do not contain a vertex of~$\Gamma$ and are not contained in a facet in $M \setminus \Sigma$.
\end{example}

\begin{example}[Closed sets of a matroid]
\label{ex:matroid}
The MacLane--Steinitz exchange property (see e.g.~\cite[Lemma~1.4.2]{Oxley1992}) ensures that the closed sets of a matroid form a locally branched finite lattice.
\end{example}

\begin{example}[Locally branched lattices with non-trivial intersection]\label{ex:unions_of_lattices}
    Let $\P_1,\dots,\P_k$ be finite locally branched meet semi-sublattices of~$B_n$.
    Then the iterator may iterate through all elements of their union by first iterating through $\P_1$, then through all elements in $\P_2$ not contained in $\P_1$ and so on.
\end{example}

\begin{example}[Polyhedral complexes]
\label{ex:polyhedral_complex}
  Using the iteration as in the previous example allows to iterate through polytopal or polyhedral complexes, or through complexes of tight spans.
\end{example}

\section{The algorithm}
\label{sec:algorithm}

Let~$\P$ be a finite locally branched lattice given as a meet semi-sublattice of the boolean lattice~$B_n$.
This is, $\At{\P} = \{1,\dots,n\}$ and $p = \At(p) \in \P$.
The following algorithm is a recursively defined depth-first iterator through the elements of~$\P$.
Given $c \in \P$ and its lower covers $x_1,\dots,x_k$, the iterator yields~$c$ and then computes, one after the other, the lower covers of $x_1,\dots,x_k$, taking into account those to be ignored, and then recursively proceeds.
Being an \emph{iterator} means that the algorithm starts with only assigning the input to the respective variables and then waits in its current state.
Whenever an output is requested, it starts from its current state and runs to the point \textbf{ITERATOR OUTPUT}, outputs the given output, and again waits.
Iterators are regularly used in modern programming languages\footnote{See \url{https://en.wikipedia.org/wiki/Iterator}.}.

\bigskip

\lstset{language    = Python,
        linewidth   = 0.95\textwidth,
        xleftmargin = 0.05\textwidth}
\begin{lstlisting}
(*Algorithm \textbf{FaceIterator} *)

(*\textbf{INPUT}*)
  (* $\bullet$ *) coatoms (* \hspace{30pt} -- list of coatoms of~$\P$ not contained in any of *)ignored_sets
  (* $\bullet$ *) ignored_sets (* \hspace{2.5pt} -- list of subsets of $\{1,\dots,n\}$*)
  (* $\bullet$ *) ignored_atoms (* -- subset of $\{1,\dots,n\}$*)

(*\textbf{PROCEDURE}*)
if coatoms (*\!\! $\neq\ \emptyset$ *): (*\label{line:loop}*)
    a := coatoms.(*\textbf{first\_element}*)() (* \label{line:pick_a} *)
    if a (*$\cap$*) ignored_atoms (*$=\emptyset$*):(*\label{line:avoid_ignored_atoms}*)
        (*\textbf{ITERATOR OUTPUT}*)
          (* $\bullet$ *) a (*\label{line:iteratoroutput}*)

    new_coatoms = (*$\{$*) a (*$\cap$*) b : b (*$\in$*) coatoms (*$\setminus$*) a (*$\}$*)(*\label{line:new_coatoms1}*)
    new_coatoms = (*$\{$*) x (*$\in$*) new_coatoms : x (*$\not\subseteq$*) y (* for all *) y (*$\in$*) ignored_sets (*$\}$*)(*\label{line:new_coatoms2}*)
    new_coatoms = new_coatoms.(*\textbf{inclusion\_maximals}*)()(*\label{line:inclusion_maximal}*)(*\label{line:new_coatoms3}*)
    (*\textbf{FaceIterator}*)( coatoms (*\hspace{31.5pt}*)= new_coatoms, (*\label{line:visit_below}*)
                ignored_sets (*\hspace{5.5pt}*)= ignored_sets.(*\textbf{copy}*)(),
                ignored_atoms = ignored_atoms )

    next_coatoms = (*$\{$*) x (*$\in$*) coatoms : x (*$\not\subseteq$*) a (*$\cup$*) ignored_atoms (*$\}$*)(*\label{line:remove_some_coatoms}*)
    ignored_sets.(*\textbf{append}*)(a (*$\cup$*) ignored_atoms)(*\label{line:append_ignored_sets}*)
    (*\textbf{FaceIterator}*)( coatoms (*\hspace{31.5pt}*)= next_coatoms, (*\label{line:visit_next}*)
                ignored_sets (*\hspace{5.5pt}*)= ignored_sets,
                ignored_atoms = ignored_atoms )
\end{lstlisting}

For polyhedra, a slightly more sophisticated version of this algorithm is implemented in \sage\footnotemark[1].
Before proving the correctness of the algorithm, we provide several detailed examples.
If not mentioned otherwise, we do not ignore any atoms and always set \codestyle{ignored\_atoms} = $\{\}$ in the examples.
We also assume the lists to be ordered lexicographically for iteration.
One may assume that the algorithm additionally visits the upper bound given by the union of the coatoms, whenever this is suitable.

\begin{example}[Square]
    We apply the algorithm to visit faces of a square.%

    \begin{itemize}
        \item \textbf{INPUT}: \codestyle{coatoms} = $[\{1,2\}, \{1,4\}, \{2,3\}, \{3,4\}]$, \codestyle{ignored\_sets} = $[]$
        \item \codestyle{a} = $\{1,2\}$, \textbf{ITERATOR OUTPUT}: $\{1,2\}$
        \item \codestyle{new\_coatoms} = $[\{1\}, \{2\}]$
        \item Apply \textbf{FaceIterator} to sublattice $[\hat{0},\{1,2\}]$
            \begin{itemize}
                \item \textbf{INPUT}: \codestyle{coatoms} = $[\{1\}, \{2\}]$, \codestyle{ignored\_sets} = $[]$
                \item \codestyle{a} = $\{1\}$, \textbf{ITERATOR OUTPUT}: $\{1\}$
                \item \codestyle{new\_coatoms} = $[\emptyset]$
                \item Apply \textbf{FaceIterator} to sublattice $[\hat{0},\{1\}]$
                    \begin{itemize}
                        \item \textbf{INPUT}: \codestyle{coatoms} = $[\emptyset]$, \codestyle{ignored\_sets} = $[]$
                        \item \codestyle{a} = $\emptyset$, \textbf{ITERATOR OUTPUT}: $\emptyset$
                        \item (\codestyle{new\_coatoms} is empty)
                        \item Apply \textbf{FaceIterator} to sublattice $[\hat{0},\hat{0}]$ without output
                        \item Add $\emptyset$ to \codestyle{ignored\_sets} (to the copy in this call of \codestyle{FaceIterator})
                        \item Reapply \textbf{FaceIterator} to sublattice $[\hat{0},\{1\}]$
                        \item \textbf{INPUT}: \codestyle{coatoms} = $[]$, \codestyle{ignored\_sets} = $[\emptyset]$
                    \end{itemize}
                \item \codestyle{ignored\_sets} = $[\{1\}]$
                \item Reapply \textbf{FaceIterator} to sublattice $[\hat{0}, \{1,2\}]$
                \item \textbf{INPUT}: \codestyle{coatoms} = $[\{2\}]$, \codestyle{ignored\_sets} = $[\{1\}]$
                \item \codestyle{a} = $\{2\}$, \textbf{ITERATOR OUTPUT}: $\{2\}$
                \item Apply \textbf{FaceIterator} to sublattice $[\hat{0},\{2\}]$
                    \begin{itemize}
                        \item \textbf{INPUT}: \codestyle{coatoms} = $[]$, \codestyle{ignored\_sets} = $[\{1\}]$
                    \end{itemize}
                \item \codestyle{ignored\_sets} = $[\{1\},\{2\}]$
                \item Reapply \textbf{FaceIterator} to sublattice $[\hat{0},\{1,2\}]$
                \item \textbf{INPUT}: \codestyle{coatoms} = $[]$, \codestyle{ignored\_sets} = $[\{1\},\{2\}]$
            \end{itemize}
        \item \codestyle{ignored\_sets} = $[\{1,2\}]$
        \item Reapply \textbf{FaceIterator} to entire lattice
        \item \textbf{INPUT}: \codestyle{coatoms} = $[\{1,4\}, \{2,3\}, \{3,4\}]$, \codestyle{ignored\_sets} = $[\{1,2\}]$
        \item \codestyle{a} = $\{1,4\}$, \textbf{ITERATOR OUTPUT}: $\{1,4\}$
        \item Apply \textbf{FaceIterator} to sublattice $[\hat{0},\{1,4\}]$
            \begin{itemize}
                \item \textbf{INPUT}: \codestyle{coatoms} = $[\{4\}]$, \codestyle{ignored\_sets} = $[\{1,2\}]$
                \item \codestyle{a} = $\{4\}$, \textbf{ITERATOR OUTPUT}: $\{4\}$
                \item Apply \textbf{FaceIterator} to sublattice $[\hat{0},\{4\}]$ without output
            \end{itemize}
        \item \codestyle{ignored\_sets} = $[\{1,2\}, \{1,4\}]$

        \vspace*{15pt}

        \item ... further outputs: $\{2,3\}$, $\{3\}$, $\{3,4\}$
    \end{itemize}
\end{example}

\clearpage
\begin{example}[Minimal triangulation of $\RP^2$]\label{ex:rp2}\
    \begin{center}
        \begin{tikzpicture}[scale=0.9]
            \tikzstyle{point} = [circle, thick, draw=black, fill=black, scale=0.4]
            \fill[fill=black!20] (0,0) -- (4,0) -- (2,3.46) -- (0,0);
            \node (1)[point, label=below:{$1$}] at (0,0){};
            \node (2)[point, label=below:{$2$}] at (2,0){};
            \node (3)[point, label=below:{$3$}] at (4,0){};
            \node (1a)[point, label=right:{$1$}] at (3,1.73){};
            \node (2a)[point, label=above:{$2$}] at (2,3.46){};
            \node (3a)[point, label=left:{$3$}] at (1,1.73){};
            \node (4)[point, label=below:{$4$}] at (1,0.5){};
            \node (5)[point, label=below:{$5$}] at (3,0.5){};
            \node (6)[point, label=above left:{$6$}] at (2,2.23){};

            \draw (1) -- (3) -- (2a) -- (1) -- (4) -- (2) -- (5) -- (3) -- (5) -- (1a) -- (6) -- (2a) -- (6) -- (3a) -- (4) -- (5) -- (6) -- (4);
        \end{tikzpicture}
    \end{center}

    \begin{itemize}
        \item \textbf{INPUT}: \codestyle{coatoms} = $[\{1,2,4\}, \dots, \{4,5,6\}]$,
            \codestyle{ignored\_sets} = $[]$
        \item \codestyle{a} = $\{1,2,4\}$, \textbf{ITERATOR OUTPUT}: $\{1,2,4\}$, $\{1,2\}$, $\{1\}$, $\emptyset$, $\{2\}$, $\{1,4\}$, $\{4\}$, $\{2,4\}$
        \item \codestyle{a} = $\{1,2,6\}$, \textbf{ITERATOR OUTPUT}: $\{1,2,6\}$, $\{1,6\}$, $\{6\}$, $\{2,6\}$
        \item \codestyle{a} = $\{1,3,4\}$, \textbf{ITERATOR OUTPUT}: $\{1,3,4\}$, $\{1,3\}$, $\{3\}$, $\{3,4\}$
        \item \codestyle{a} = $\{1,3,5\}$, \textbf{ITERATOR OUTPUT}: $\{1,3,5\}$, $\{1,5\}$, $\{5\}$, $\{3,5\}$
        \item \codestyle{a} = $\{1,5,6\}$, \textbf{ITERATOR OUTPUT}: $\{1,5,6\}$, $\{5,6\}$
        \item \codestyle{a} = $\{2,3,5\}$, \textbf{ITERATOR OUTPUT}: $\{2,3,5\}$, $\{2,3\}$, $\{2,5\}$
        \item \codestyle{a} = $\{2,3,6\}$, \textbf{ITERATOR OUTPUT}: $\{2,3,6\}$, $\{3,6\}$
        \item \codestyle{a} = $\{2,4,5\}$, \textbf{ITERATOR OUTPUT}: $\{2,4,5\}$, $\{4,5\}$
        \item \codestyle{a} = $\{3,4,6\}$, \textbf{ITERATOR OUTPUT}: $\{3,4,6\}$, $\{4,6\}$
        \item \codestyle{a} = $\{4,5,6\}$, \textbf{ITERATOR OUTPUT}: $\{4,5,6\}$
    \end{itemize}
\end{example}

\begin{example}[Tight span]\

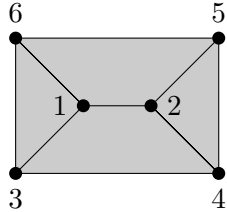
    \begin{center}
        \begin{tikzpicture}[scale=0.9]
            \tikzstyle{point} = [circle, thick, draw=black, fill=black, scale=0.4]
            \fill[fill=black!20] (0,0) -- (3,0) -- (3,2) -- (0,2);
            \node (3)[point, label=below:{$3$}] at (0,0){};
            \node (4)[point, label=below:{$4$}] at (3,0){};
            \node (5)[point, label=above:{$5$}] at (3,2){};
            \node (6)[point, label=above:{$6$}] at (0,2){};
            \node (1)[point, label=left:{$1$}] at (1,1){};
            \node (2)[point, label=right:{$2$}] at (2,1){};

            \draw (3) -- (4) -- (5) -- (6) -- (3) -- (1) -- (6) -- (1) -- (2) -- (4) -- (2) -- (5);
        \end{tikzpicture}
    \end{center}
    \begin{itemize}
        \item \textbf{INPUT}: \codestyle{coatoms} = $[\{1,2,3,4\}, \{1,2,5,6\}, \{1,3,6\}, \{2,4,5\}]$, \newline
            \codestyle{ignored\_sets} = $[\{3,4,5,6\}]$, \codestyle{ignored\_atoms} = $\{3,4,5,6\}$
        \item \codestyle{a} = $\{1,2,3,4\}$, no output of $\{1,2,3,4\}$
        \item \codestyle{new\_coatoms} = $\{\{1,2\}, \{1,3\}, \{2,4\}\}$
            \begin{itemize}
                \item \textbf{INPUT}: \codestyle{coatoms} = $[\{1,2\}, \{1,3\}, \{2,4\}]$,\newline
                    \codestyle{ignored\_sets} = $[\{3,4,5,6\}]$, \codestyle{ignored\_atoms} = $\{3,4,5,6\}$
                \item \codestyle{a} = $\{1,2\}$, \textbf{ITERATOR OUTPUT}: $\{1,2\}$, $\{1\}$, $\{2\}$
                \item \codestyle{ignored\_sets} = $[\{3,4,5,6\}, \{1,2,3,4,5,6\}]$, \codestyle{coatoms} = $\{\}$
            \end{itemize}
        \item \codestyle{ignored\_sets} = $[\{3,4,5,6\}, \{1,2,3,4,5,6\}]$, \codestyle{coatoms} = $\{\}$
    \end{itemize}
\end{example}

\clearpage

\begin{example}[Polyhedral complex]\

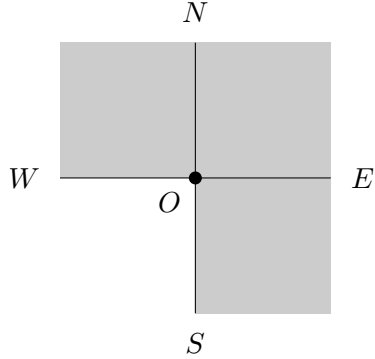
    \begin{center}
        \begin{tikzpicture}[scale=0.9]
            \tikzstyle{point} = [circle, thick, draw=black, fill=black, scale=0.4]
            \fill[fill=black!20] (0,0) -- (2,0) -- (2,2) -- (0,2) -- (0,0);
            \fill[fill=black!20] (0,0) -- (-2,0) -- (-2,2) -- (0,2) -- (0,0);
            \fill[fill=black!20] (0,0) -- (2,0) -- (2,-2) -- (0,-2) -- (0,0);

            \node (0)[point, label=below left:{$O$}] at (0,0){};
            \node (W)[label=left:{$W$}] at (-2,0){};
            \node (E)[label=right:{$E$}] at (2,0){};
            \node (N)[label=above:{$N$}] at (0,2){};
            \node (S)[label=below:{$S$}] at (0,-2){};

            \draw (0,0) -- (2,0);
            \draw (0,0) -- (-2,0);
            \draw (0,0) -- (0,2);
            \draw (0,0) -- (0,-2);
        \end{tikzpicture}
    \end{center}
    \emph{Incorrect application} by applying to all polyhedra in the complex:
    \begin{itemize}
        \item \textbf{INPUT}: \codestyle{coatoms} = $[\{W,N,0\}, \{N,E,0\}, \{S,E,0\}]$, \codestyle{ignored\_sets} = $[]$
        \item \codestyle{a} = $\{W,N,0\}$, \textbf{ITERATOR OUTPUT}: $\{W,N,0\}$
        \item \codestyle{new\_coatoms} = $[\{N,0\}]$, \textbf{ITERATOR OUTPUT}: $\{N,0\}$
        \item \codestyle{ignored\_sets} = $[\{W,N,0\}]$
        \item \codestyle{a} = $\{N,E,0\}$, \textbf{ITERATOR OUTPUT}: $\{N,E,0\}$
        \item \codestyle{new\_coatoms} = $[\{E,0\}]$, \textbf{ITERATOR OUTPUT}: $\{E,0\}$
        \item \codestyle{ignored\_sets} = $[\{W,N,0\}, \{N,E,0\}]$
        \item \codestyle{a} = $\{S,E,0\}$, \textbf{ITERATOR OUTPUT}: $\{S,E,0\}$
        \item \codestyle{new\_coatoms} = $[]$
    \end{itemize}
    \emph{Correct application} by applying successively to all faces of all polyhedra:
    \begin{itemize}
        \item \textbf{ITERATOR OUTPUT}: $\{W,N,0\}$ (output of $\hat{1}$ before applying \textbf{FaceIterator})
        \item Apply algorithm for $\{W,N,0\}$:
            \begin{itemize}
                \item \textbf{INPUT}: \codestyle{coatoms} = $[\{W,0\}, \{N,0\}]$, \codestyle{ignored\_sets} = $[\{W,N\}]$
                \item \textbf{ITERATOR OUTPUT}: $\{W,0\}$, $\{0\}$, $\{N,0\}$
            \end{itemize}
        \item \textbf{ITERATOR OUTPUT}: $\{N,E,0\}$
        \item Apply algorithm for $\{N,E,0\}$:
            \begin{itemize}
                \item \textbf{INPUT}: \codestyle{coatoms} = $[\{E,0\}]$, \codestyle{ignored\_sets} = $[\{W,N,0\}, \{N,E\}]$
                \item \textbf{ITERATOR OUTPUT}: $\{E,0\}$
            \end{itemize}
        \item \textbf{ITERATOR OUTPUT}: $\{S,E,0\}$
        \item Apply algorithm for $\{S,E,0\}$:
            \begin{itemize}
                \item \textbf{INPUT}: \codestyle{coatoms} = $[\{S,0\}]$, \codestyle{ignored\_sets} = $[\{W,N,0\}, \{N,E,0\}, \{S,E\}]$
                \item \textbf{ITERATOR OUTPUT}: $\{S,0\}$
            \end{itemize}
    \end{itemize}
\end{example}

\subsection{Correctness of the algorithm}

As assumed, let~$\P$ be a locally branched meet semi-sublattice of the boolean lattice $B_n$.
In the following properties and their proofs, we indeed see that if~$\P$ is any meet semi-sublattice of~$B_n$, the algorithm visits exactly once each element $p \in \P$ not contained in any of \codestyle{ignored\_sets} and not containing any of \codestyle{ignored\_atoms} if the interval $[p, \hat{1}]$ is locally branched.

\begin{proposition}
  The algorithm \emph{\textbf{FaceIterator}} is well-defined in the following sense:
  Let $a \in $ \codestyle{coatoms}.
  Then
  \begin{enumerate}
    \item The call of \emph{\textbf{FaceIterator}} in line~\ref{line:visit_below} applies the algorithm to the sublattice $[\hat 0,a]$. \label{it:visit_below}
    \item The call of \emph{\textbf{FaceIterator}} in line~\ref{line:visit_next} applies the algorithm to~$\P$ with $a\ \cup $ \codestyle{ignored\_atoms} appended to \codestyle{ignored\_sets} and all coatoms contained in $a\ \cup $ \codestyle{ignored\_atoms} removed. \label{it:visit_next}
  \end{enumerate}
\end{proposition}

\begin{proof}
  The proof of~\ref{it:visit_next} is obvious.
  To prove~\ref{it:visit_below}, we have to show that the construction of \codestyle{new\_coatoms} in lines~\ref{line:new_coatoms1}--\ref{line:new_coatoms3} is correct.
  First, observe that all elements in \codestyle{new\_coatoms} are strictly below the element~$a$.
  Next, let $x \prec a \prec \hat 1$ in~$\P$.
  Since~$\P$ is locally branched there is an element $b \neq a$ with $ x < b \prec \hat 1$, implying $x = a \cap b$.
  If~$x$ is not contained in any element in \codestyle{ignored\_sets}, then the same holds for~$b$ and thus, $b \in$ \codestyle{coatoms} and $x \in$ \codestyle{new\_coatoms}.
  This implies that \codestyle{new\_coatoms} are exactly the lower covers of~$a$ not contained in an element of \codestyle{ignored\_sets}, as desired.
\end{proof}

\begin{theorem}
\label{thm:algorithm}
  The algorithm \emph{\textbf{FaceIterator}} iterates exactly once over all element in~$\P$ which are not contained in any subset of \codestyle{ignored\_sets}, and do not contain any element in \codestyle{ignored\_atoms}.
\end{theorem}

\begin{proof}
  We argue by induction on the cardinality of \codestyle{coatoms}.
  First note that the cardinalities of \codestyle{new\_coatoms} and \codestyle{next\_coatoms} in the two subsequent calls of \textbf{FaceIterator} in lines~\ref{line:visit_below} and~\ref{line:visit_next} are both strictly smaller than the cardinality of \codestyle{coatoms}.
  If \codestyle{coatoms} $=\emptyset$, then all elements of~$\P \setminus \hat 1$ are contained in elements of \codestyle{ignored\_sets}, and the algorithm does correctly not output any element.
  Suppose that \codestyle{coatoms} $\neq\emptyset$ and let $a$ be its first element assigned in line~\ref{line:pick_a}.
  Let $p \in \P$.
  If~$p$ is contained in an element of \codestyle{ignored\_sets} or contains an element in \codestyle{ignored\_atoms} then it is not outputed by the algorithm.
  Otherwise,
  \begin{itemize}
    \item if $p = a$, then the algorithm outputs~$p$ correctly in line~\ref{line:iteratoroutput},
    \item if $p < a$, then~$p$ is outputed in the call of \textbf{FaceIterator} in line~\ref{line:visit_below} by induction,
    \item if $p \not\leq a$, then~$p \not\leq a\ \cup$ \codestyle{ignored\_atoms} and~$p$ is outputed in the call of \textbf{FaceIterator} in line~\ref{line:visit_next} by induction.
    \qedhere
  \end{itemize}
\end{proof}

\subsection{Variants of the algorithm}

We finish this section with a dualization property followed by explicitly stating the result when applying the algorithm for the variants discussed in \Cref{sec:frameworkvariants}.

\begin{corollary}
\label{cor:algorithm_opposite}
  If \codestyle{ignored\_sets} is a list of coatoms, then the algorithm can also be applied to~$\Pop$ with the roles of atoms and coatoms interchanged.
\end{corollary}

We later see in \Cref{thm:runtime_algorithm} that considering~$\Pop$ instead of~$\P$ might be faster as the runtime depends on the number of coatoms.
For example, in \Cref{ex:rp2} one could apply the algorithm to $\P^{\op}$ to improve runtime as there are $10$ facets but only $6$ vertices.

\begin{corollary}
\label{cor:algorithm_polytopes}
  Let~$P$ be a polytope.
  Provided the vertex-facet incidences of~$\P$, the above algorithm visits exactly once all faces of~$P$.
\end{corollary}

\begin{corollary}
\label{cor:algorithm_polyhedra}
  Let~$P$ be an unbounded polyhedron and let $\overline{P}$ be a projectively equivalent polytope with marked face.
  Provided the vertex-facet incidences of $\overline{P}$, the above algorithm visits exactly once all non-empty faces of~$P$.
\end{corollary}

Actually, a non-empty intersection of two faces of $P$ is not contained in the marked facet at infinity.
Hence, one could even use the algorithm without providing the marked facet at infinity.
This might or might not visit the empty face.

\begin{corollary}
\label{cor:algorithm_manifold}
  Let~$P$ be a finite polytopal subidivision of a closed manifold.
  Provided the vertex-facet incidences the above algorithm visits exactly once all faces of~$P$.
\end{corollary}

\begin{corollary}
\label{cor:algorithm_tight_spans}
  Let~$\Sigma$ be an extended tight span in $\R^d$ as described in \Cref{ex:tight_spans}.
  Let~$V$ be a subset of vertices of~$\Sigma$.
  Provided the facets and the boundary faces in their vertex description, the above algorithm visits exactly once all faces of~$\Sigma$ not containing any vertex of~$V$.
\end{corollary}

\begin{corollary}
\label{cor:algorithm_polyhedral_complex}
  Let~$P$ be a polyhedral complex.
  Provided each maximal cell as vertex-facet incidences with possibly marked face at infinity.
  The algorithm can be applied to visit exactly once each element in~$P$.
\end{corollary}

\section{Data structures, memory usage, and theoretical runtime}
\label{sec:data_structures}

The operations used in the algorithm are \codestyle{intersetions}, \codestyle{subset checks} and \codestyle{unions}.
It will turn out that the crucial operation for the runtime is the subset check.

\medskip

For the theoretical runtime we consider representation as (sparse) \emph{sorted-lists-of-atoms}.
However, in the implementation we use (dense) \emph{atom-incidence-bit-vectors}.
This is theoretically slighly slower, but the crucial operations can all be done using bitwise operations.

\medskip

Observe that a \emph{sorted-lists-of-atoms} needs as much memory as there are incidences.
Consider two sets~$A$ and~$B$ (of integers) of lengths~$a$ and~$b$, respectively, and a (possibly unsorted) list~$C$ of $m$ sets of total length $\alpha$.
Using standard implementations, we assume in the runtime analysis that
\begin{itemize}
    \item finding (and possibly deleting) a given element~$x$ inside~$C$ has runtime~$\O(\alpha)$,
    \item deleting all duplicates in~$C$ has runtime~$\O(\alpha \cdot m)$,
    \item intersection $A \cap B$ and union $A \cup B$ have runtime $\O(\max(a,b))$,
    \item a subset check $A \subseteq B$ has runtime $\O(b)$ and
    \item to check whether $A$ is subset of any element in $C$ has runtime $\O(\alpha)$.
\end{itemize}
Let~$r+1$ be the number of elements in a longest chain in~$\P$, let~$m = |\text{\codestyle{coatoms}}|$, $n = |\text{\codestyle{atoms}}|$, and let
\[
  \alpha = \sum_{a \in \text{\codestyle{coatoms}}\cup \text{\codestyle{ignored\_sets}}} |a \cup\text{\codestyle{ignored\_atoms}}|.
\]
(In the case that \codestyle{ignored\_sets} and \codestyle{ignored\_atoms} are both empty, $\alpha$ is the total number of atom-coatom incidences.)
Moreover, let~$\varphi$ be the number of recursive calls of the algorithm.
(In the case that \codestyle{ignored\_atoms} is empty, $\varphi$ is the cardinality of $\P$.
Otherwise, it is bounded by this cardinality.)

\begin{theorem}\label{thm:runtime_algorithm}
    The algorithm has memory consumption $\O(\alpha \cdot r)$ and runtime $\O(\alpha \cdot m \cdot \varphi)$.
\end{theorem}
\begin{proof}
    Note first that at each recursive call of \textbf{FaceIterator} the number of \codestyle{coatoms} is bounded by $m$ and the total length of \codestyle{coatoms}, \codestyle{ignored\_sets} and \codestyle{ignored\_atoms} is bounded each by~$\alpha$.
    With above assumptions, \textbf{FaceIterator} has runtime $\O(\alpha \cdot m)$ not considering recursive calls.

    A single call of \textbf{FaceIterator} has memory usage at most~$c \cdot \alpha$ for a global constant~$c$, not taking into account the recursive calls.
    The call in line~\ref{line:visit_next} does not need extra memory as all old variables can be discarded.
    The longest chain of the lattice $[0,\text{\codestyle{a}}]$ is at most of length $r-1$.
    By induction the call of \textbf{FaceIterator} in line~\ref{line:visit_below} has total memory consumption at most $(r-1) \cdot c \cdot \alpha$.
    The claimed bounds follow.
\end{proof}

\subsection{Computing all cover relations}
\label{sec:covers}

Applying the algorithm to a graded locally branched meet semi-sublattice of $B_n$ while keeping track of the recursion depth allows an a posteriori sorting of the output by the level sets of the grading.
The recursion depth is the number of iterative calls using line~\ref{line:visit_below}.
We obtain the same bound for generating all cover relations as V.~Kaibel and M.~E.~Pfetsch~\cite{Kaibel2002}.
For this we additionally assume that
\begin{itemize}
  \item a list of $\varphi$ sets each of length at most $n$ can be sorted in time $\O(n \cdot \varphi \cdot \log \varphi)$ and
  \item a set with $a$ elements can be looked up in a sorted list of $\varphi$ sets in time $\O(a \cdot \log \varphi)$.
\end{itemize}
\begin{proposition}\label{prop:all_cover_relations}
  Let $\P$ be a graded meet sublattice of $\B_n$.
  Assume each level set of $\P$~to be given as \emph{sorted-lists-of-atoms}, one can generate all cover relations in time $\O(\alpha \cdot \min(m,n) \cdot \varphi)$ with quantities as defined above using the above algorithm.
\end{proposition}

Observe that in the situation of this proposition, \codestyle{ignored\_sets} and \codestyle{ignored\_atoms} are both empty and in particular~$\alpha$ is the total length of the coatoms.

\begin{proof}
  First, we sort all level-sets.
  Then, we intersect each element with each coatom, obtaining its lower covers and possibly other elements.
  We look up each intersection to determine the lower covers.

  Sorting the level sets is done in time $\O(n \cdot \varphi \cdot \log \varphi)$.
  All such intersection are obtained in time $\O(\varphi \cdot m \cdot n)$.
  For a fixed element the total length of its intersections with all coatoms is bounded by $\alpha$.
  Hence, all lookups are done in time $\O(\varphi \cdot \alpha \cdot \log \varphi)$.

  Finally, we note that $m,n \leq \alpha$ and that $\log \varphi \leq \min(m,n)$.
\end{proof}

In the ungraded case, one first sorts all elements in $\P$, and then intersects each element~$p$ with all coatoms.
The inclusion maximal elements among those strictly below~$p$ are lower covers of~$p$.
They can be looked up in the list of sorted elements to obtain an index.
Observe that all this is done time~$\O(\alpha \cdot m \cdot \varphi)$.

\section{Performance of the algorithm implemented in \sage}
\label{sec:performance}

We present running times for the several computations.
These are performed on an Intel$^\text{\tiny{\textregistered}}$ Core\texttrademark{} i7-7700 CPU @ 3.60GHz x86\_64-processor with 4 cores and 30 GB of RAM.
The computations are done either using
\begin{itemize}
    \item \polymake~\texttt{3.3}~\cite{polymake:2000}, or
    \item \normaliz~\texttt{3.7.2}~\cite{Normaliz}, or
    \item the presented algorithm in \sageversion, or
    \item the presented algorithm in \sageversion\ with additional parallelization,
        intrinsics, and subsequent improvements\footnote{
            The presented algorithm can be parallelized easily by altering the loop call in line~\ref{line:loop} on page \pageref{line:loop} in \Cref{sec:algorithm}.
            The implementation using bitwise operations also allows to use instructions for intrinsics such as \emph{Advanced Vector Extensions}.
            This results in a runtime improvement of at least a factor~$2$.
            Also, the subset check has been improved since \sageversion.
            All these improvements will be made available in \sage{}, see \url{https://trac.sagemath.org/ticket/28893} and \url{https://trac.sagemath.org/ticket/27103}.
        }
\end{itemize}
The default algorithms in \sage\ before version \sageversion\ performs much worse than either of these and is not considered here.
Given the vertex-facet incidences, we computed
\begin{figure}[htbp]
    \begin{tikzpicture}%
    \begin{loglogaxis}[
      axis x line=bottom,
      axis y line=left,
      axis equal image,
      xmin=0.01,
      xmax=1000000,
      ymin=0.01,
      ymax=1000000,
      xlabel={\small\polymake},
      ylabel={\tiny cover relations, {\color{red}$f$-vector}, {\color{blue}improved $f$-vector}}
      ]
      ]
      \addplot[only marks, mark options={scale=0.5}] table[x=polymake,y=cover,col sep=comma] {data.csv};
      \addplot[only marks, color=red, mark options={scale=0.5}] table[x=polymake,y=new,col sep=comma] {data.csv};
      \addplot[only marks, color=blue, mark options={scale=0.5}] table[x=polymake,y=newparallel,col sep=comma] {data.csv};
      \addplot [dashed,color=gray,mark=none,domain=0.01:1000000] {x};
      \addplot [dashed,color=gray,mark=none,domain=0.01:1000000] {0.1*x};
      \addplot [dashed,color=gray,mark=none,domain=0.01:1000000] {0.01*x};
      \addplot [dashed,color=gray,mark=none,domain=0.01:1000000] {0.001*x};
      \addplot [dashed,color=gray,mark=none,domain=0.01:1000000] {0.0001*x};
      \addplot [dashed,color=gray,mark=none,domain=0.01:1000000] {0.00001*x};
    \end{loglogaxis}
  \end{tikzpicture}
  \hfill
  \begin{tikzpicture}%
    \begin{loglogaxis}[
      axis x line=bottom,
      axis y line=left,
      axis equal image,
      xmin=0.01,
      xmax=1000000,
      ymin=0.01,
      ymax=1000000,
      xlabel={\small\normaliz},
      ylabel={\tiny{\color{red}$f$-vector}, {\color{blue}improved $f$-vector}}
      ]
      ]
      \addplot[only marks, color=red, mark options={scale=0.5}] table[x=normaliz,y=new,col sep=comma] {data.csv};
      \addplot[only marks, color=blue, mark options={scale=0.5}] table[x=normaliz,y=newparallel,col sep=comma] {data.csv};
      \addplot [dashed,color=gray,mark=none,domain=0.01:1000000] {x};
      \addplot [dashed,color=gray,mark=none,domain=0.01:1000000] {0.1*x};
      \addplot [dashed,color=gray,mark=none,domain=0.01:1000000] {0.01*x};
      \addplot [dashed,color=gray,mark=none,domain=0.01:1000000] {0.001*x};
      \addplot [dashed,color=gray,mark=none,domain=0.01:1000000] {0.0001*x};
      \addplot [dashed,color=gray,mark=none,domain=0.01:1000000] {0.00001*x};
    \end{loglogaxis}
  \end{tikzpicture}

  \caption{Comparision of the runtimes.  Every dot represents one best-of-five computation, and every shifted diagonal is a factor-$10$ faster runtime. Dots at the boundary represent memory overflows.}
\end{figure}
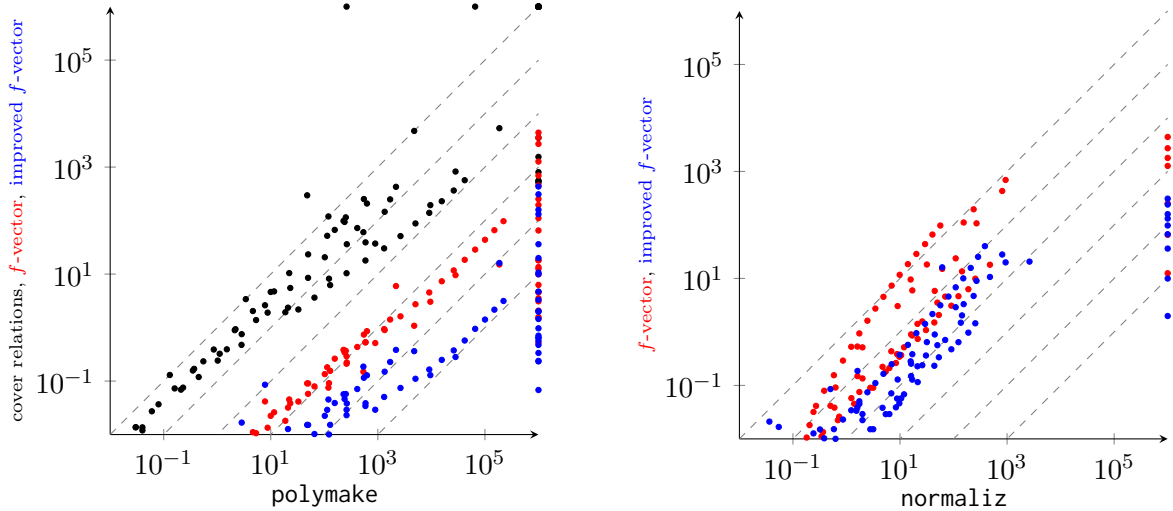

\setlist[enumerate,1]{label={(\arabic*)}}
\begin{enumerate}
  \item\label{it:comp1} cover relations and $f$-vector in \polymake,
  \item\label{it:comp6} $f$-vector in \normaliz\ \emph{with} parallelization,
  \item\label{it:comp4} $f$-vector with the presented implementation in \sage,
  \item\label{it:comp3} all cover relations with the presented implementation in \sage,
  \item\label{it:comp5} $f$-vector with the presented implementation in \sage\ \emph{with} parallelization, intrinsics and additional improvements.
\end{enumerate}
We remark that
\begin{itemize}
  \item the computation of the $f$-vector in~\ref{it:comp1} also calculates all cover relations,
  \item \polymake\ also provides a different algorithm to compute the $f$-vector from the $h$-vector for simplicial/simple polytopes (providing this additional information sometimes improves the performance in \polymake), and
  \item \normaliz\ does not provide an algorithm to compute the cover relations.
\end{itemize}
For every algorithm we record the best-of-five computation on
\begin{itemize}
    \item the simplex of dimension~$n$,
    \item several instances of the cyclic polyhedron of dimension $10$ and $20$,
    \item the associahedron of dimension~$n$,
    \item the permutahedron of dimension~$n$ embedded in dimension $n+1$,
    \item a $20$-dimensional counterexample to the Hirsch-conjecture,
    \item the cross-polytope of dimension~$n$,
    \item the Birkhoff-polytope of dimension $(n-1)^2$,
    \item joins of such polytopes with their duals,
    \item Lawrence polytopes of such polytopes,
\end{itemize}
and refer to \Cref{sec:runtimes} for the detailed runtimes.

\section{Application of the algorithm to Wilf's conjecture}
\label{sec:wilf}

W.~Bruns, P.~García-Sánchez, C.~O'Neill and D.~Wilburne provided an algorithm that verifies Wilf's conjecture for a given fixed multiplicity~\cite{Bruns}.
We give a brief overview of their approach:

\begin{definition}
  A \defn{numerical semigroup} is a set $S \subset \Z_{\geq{0}}$ containing $0$ that is closed under addition and has finite complement.
  \begin{itemize}
      \item Its \defn{conductor} $c(S)$ is the smallest integer $c$ such that $c + \Z_{\geq 0} \subseteq S$.
      \item Its \defn{sporadic elements} are the elements $a \in S$ with $a < c(S)$ and let $n(S)$ be the number of sporadic elements.
      \item The \defn{embedding dimension} $e(S) = |S \setminus (S+S)|$ is the number of elements that cannot be written as sum of two elements.
      \item The \defn{multiplicity} $m(S)$ is the minimal nonzero element in~$S$.
  \end{itemize}
\end{definition}

\begin{conjecture}[Wilf]
    For any numerical semigroup $S$,
    \[ c(S) \leq e(S)n(S).\]
\end{conjecture}

For fixed mulitplicity $m$ one can analyse certain polyhedra to verify this conjecture.

\begin{definition}[{\cite[Def.~3.3]{Bruns}}]
    Fix an integer $m\geq 3$.
    The \defn{relaxed Kunz polyhedron} is the set $P'_m$ of rational points $(x_1,\dots,x_{m-1}) \in \R^{m-1}$ satisfying
    \begin{alignat*}{3}
        x_i + x_j \geq x_{i+j} \qquad 1 \leq i \leq j \leq m-1, \quad i +j < m,\\
        x_i + x_j + 1\geq x_{i+j} \qquad 1 \leq i \leq j \leq m-1,\quad  i +j > m,
    \end{alignat*}
    The \defn{Kunz cone} is the set $C_m$ of points $(x_1,\dots,x_{m-1}) \in \R^{m-1}$ satisfying
    \[
      x_i + x_j \geq x_{i+j} \qquad 1 \leq i \leq j \leq m-1, \quad i +j \neq m.
    \]
    (All indices in this definition are taken modulo~$m$.)
\end{definition}

\begin{definition}
    Let $F$ be a face of $P'_m$ or $C_m$.
    Denote by $e(F) - 1$ and $t(F)$ the number of variables not appearing on the right and left hand sides resp.~of any defining equations of $F$.
\end{definition}

The Kunz cone is a translation of the relaxed Kunz polyhedron.
$e(F)$ and $t(F)$ are invariants of this translation.

Every numerical semigroup $S$ of multiplicty $m$ corresponds to a (all-)positive lattice point in $P'_m$.
If the point corresponding to $S$ lies in the interior of some face $F \subseteq P'_m$, then $e(F) = e(S)$ and $t(F) = t(S)$, see \cite[Thm.~3.10 \& Cor.~3.11]{Bruns}.
The following proposition summarizes the approach by which we can check for bad faces:

\begin{proposition}[\cite{Bruns}]\label{prop:checkbad}
  There exists a numerical semigroup $S$ with multiplicity $m$ that violates Wilf's conjecture if and only if
  there exists a face $F$ of $P'_m$ with positive integer point $(x_1,\dots, x_{m-1}) \in F^\circ$ and $f \in [1,m-1]$ such that
  \[
    mx_i +i \leq mx_f + f \quad \text{for every $i \neq f$}
  \]
  and
  \[
    mx_f + f - m + 1 > e(F)\cdot\big(mx_f + f - m - (x_1 + \dots, + x_{m-1}) + 1\big).
  \]
\end{proposition}

A face $F$ of $P'_m$ is \defn{Wilf} if no interior point corresponds to a violation of Wilf's conjecture.
A face $F$ of $C_m$ is \defn{Wilf}, if the corresponding face in $P'_m$ is Wilf.

\begin{proposition}[\cite{Bruns}~p.9]\label{prop:badfaces}
    Let $F$ be a face of $P'_m$ or $C_m$.
    \begin{itemize}
        \item If $e(F) > t(F)$, then $F$ is Wilf.
        \item If $2e(F) \geq m$, then $F$ is Wilf.
    \end{itemize}
\end{proposition}

Checking Wilf's conjecture for fixed multiplicity $m$ can be done as follows:
\begin{enumerate}
    \item For each face $F$ in $C_m$ check if \Cref{prop:badfaces} holds.
    \item If \Cref{prop:badfaces} does not hold, check with \Cref{prop:checkbad} if the translated face in $P'_m$ contains a point corresponding to a counterexample of the Wilf's conjecture.
\end{enumerate}
We say that a face $F$ of $C_m$ is \defn{bad} if \Cref{prop:badfaces} does not hold.
The group of units $(\Z/m\Z)^\times$ acts on $\R^{m-1}$ by multiplying indices.
The advantage of the Kunz cone over the (relaxed) Kunz polyhedron is that it is symmetric with respect to this action.
Even more, $e(F)$ and $t(F)$ are invariant under this action.
Thus in order to determine the bad faces, it suffices to determine for one representative of its orbit, if it is bad.
We say that an orbit is \defn{bad}, if all its faces are bad.

\medskip

While \cite{Bruns} uses a modified algorithm of \normaliz{} to determine all bad orbits, we replace this by the presented algorithm.
To use the symmetry of $C_m$, after visiting a facet $F$, we mark all facets in the orbit of $F$ as visited.

We visit of each orbit at least one face.
By sorting the facets by orbits and some lexicographic ordering of the facet-indices, we can even guarantee that we will visit the first element of each orbit.

Now, we are in a situation to apply the presented algorithm to Wilf's conjecture.
The concrete implementation is available in \sage{}\footnote{See \url{https://git.sagemath.org/sage.git/tree/?h=u/gh-kliem/KunzConeWriteBadFaces}.}.%
In the following table we compare the runtimes of computing the bad orbits:%

\begin{center}
  \begin{tabular}{r|r|c|c}
      m & \# bad orbits & \normaliz & \sage \\
      \hline
      15 & 180,464 & 3:33 m & 7 s \\
      16 & 399,380 & 54:39 m & 1:14 m \\
      17 & 3,186,147 & 19:35 h & 16:55 m \\
      18 & 17,345,725 & 27:13 d & 16:22 h \\
      19 & 100,904,233 & & 14:22 d
  \end{tabular}
\end{center}
These are performed on an Intel$^\text{\tiny{\textregistered}}$ Xeon\texttrademark{} CPU E7-4830 @ 2.20GHz with a total of 1 TB of RAM and 40 cores.
We used 40 threads and about 200 GB of RAM.
The timings in \cite{Bruns} used only 32 threads and a slightly slower machine.

\medskip

While testing all bad faces takes a significant amount of time, recent work by S.~Eliahou has simplified this task.

\begin{theorem}[{\cite[Thm.~1.1]{eliahou}}]
    Let $S$ be a numerical semigroup with multiplicity $m$.
    If $3e(S) \geq m$ then $S$ satisfies Wilf’s conjecture.
\end{theorem}

Checking the remaining orbits can be done quickly (we used an Intel$^\text{\tiny{\textregistered}}$ Core\texttrademark{} i7-7700 CPU @ 3.60GHz x86\_64-processor with 4 cores).
For each of the orbits with $3e < m$, we have checked whether the corresponding region is empty analogously to the computation in~\cite{Bruns}:

\begin{center}
  \begin{tabular}{r|r|r}
    m & \# orbits & time\\
    \hline
    15 & 193        & 1 s   \\
    16 & 5,669      & 11 s   \\
    17 & 7,316      & 31 s    \\
    18 & 17,233     & 1:54 m    \\
    19 & 285,684    & 2:22 h \\
  \end{tabular}
\end{center}
This computation yields the following proposition.

\begin{proposition}
  Wilf's conjecture holds for $m = 19$.
\end{proposition}

\appendix

\section{Detailed runtimes}
\label{sec:runtimes}

We give, for each of the five computations, an example of how it is executed for the $2$-simplex.
\medskip
\lstset{%
        linewidth   = 0.95\textwidth,
        xleftmargin = 0.05\textwidth,
        basicstyle  = \ttfamily,
        frame       = none,
        numbers     = none}

\ref{it:comp1} Compute cover relations and $f$-vector in \polymake\ from vertex-facet-incidences:

\begin{lstlisting}
polytope> new Polytope(VERTICES_IN_FACETS=>
                            [[0,1],[0,2],[1,2]])->F_VECTOR;
\end{lstlisting}

\ref{it:comp6} Compute $f$-vector with \normaliz
\begin{lstlisting}
sage: P = polytopes.simplex(2, backend='normaliz')
sage: P._nmz_result(P._normaliz_cone, 'FVector')
\end{lstlisting}

\ref{it:comp4} Compute cover relations in \sage{}:

\begin{lstlisting}
sage: C = CombinatorialPolyhedron([[0,1],[0,2],[1,2]])
sage: C._compute_face_lattice_incidences()
\end{lstlisting}

\ref{it:comp3} \& \ref{it:comp5} Compute $f$-vector in \sage{}
\begin{lstlisting}
sage: C = CombinatorialPolyhedron([[0,1],[0,2],[1,2]])
sage: C.f_vector()
\end{lstlisting}

For displaying the runtimes, we use the following notations:

\medskip

\begin{itemize}
\setlength{\itemsep}{8pt}
    \item $\Delta_d$ for the $d$-dimensional simplex,
    \item $\C_{d,n}$ for the $d$-dimensional cyclic polytope with $n$ vertices,
    \item $\A_d$ for the $d$-dimensional associahedron,
    \item $\P_d$ for the $d$-dimensional permutahedron,
    \item $\Hi$ for the $20$-dimensional counterexample to the Hirsch conjecture,
    \item $\square_d$ for the $d$-cube,
    \item $\B_n$ for the $(n-1)^2$-dimensional Birkhoff polytope,
    \item $P^\op$ for the polar dual of a polytope~$P$,
    \item $L(P)$ for the Lawrence polytope of~$P$.
\end{itemize}

\medskip

The runtimes of the five best-of-five computations on the various examples are as given in the following table.
``MOF'' indicates that the process was killed due to memory overflow, and a dash indicates a runtime of less than a second.

\clearpage

\begin{center}
\scalebox{0.85}{
\begin{tabular}{r||ccccc||r||ccccc}
                            & \multicolumn{5}{c||}{Time in s}                                               &                  & \multicolumn{5}{c}{Time in s}\\
  \hline
  \hline
  {}                       & \ref{it:comp1}    & \ref{it:comp6}    & \ref{it:comp4}    & \ref{it:comp3}  & \ref{it:comp5} & {}                & \ref{it:comp1}    & \ref{it:comp6}    & \ref{it:comp4}    & \ref{it:comp3} &\ref{it:comp5} \\
$\square_5 \star \square_5^\op$&   3.4&   ---&  3.4  &    ---&  ---&       $\Delta_{18}$& 1.0e1&   ---&  4.6  &   ---&   ---\\
$\square_6 \star \square_6^\op$& 1.2e2& 1.7  &  1.2e2&    ---&  ---&       $\Delta_{19}$& 2.2e1& 1.4  &  1.0e1&   ---&   ---\\
$\square_7 \star \square_7^\op$& 4.8e3& 3.1e1&  4.7e3&  1.1  &  ---&       $\Delta_{20}$& 5.0e1& 2.8  &  2.3e1&   ---&   ---\\
$\square_8 \star \square_8^\op$&   MOF& 4.8e2&    MOF&  1.8e1&1.1e1&       $\Delta_{21}$& 1.1e2& 5.4  &  5.2e1&   ---&   ---\\
          $\A_4 \star \A_4^\op$& 2.9  &   ---&    ---&    ---&  ---&       $\Delta_{22}$& 2.5e2& 1.1e1&  1.1e2&   ---&   ---\\
          $\A_5 \star \A_5^\op$& 5.4e2& 1.6  &  6.1e1&    ---&  ---&       $\Delta_{23}$& 5.5e2& 2.1e1&  2.5e2&   ---&   ---\\
          $\A_6 \star \A_6^\op$& 1.9e5& 6.1e1&  5.3e3&  1.5e1&1.6e1&       $\Delta_{24}$&   MOF& 4.4e1&  5.5e2& 1.5  &   ---\\
                        $\A_7 $& 1.0  &   ---&    ---&    ---&  ---&       $\Delta_{25}$&   MOF& 9.1e1&       & 3.1  &   ---\\
                        $\A_8 $& 2.1e1&   ---&  2.4  &    ---&  ---&       $\Delta_{26}$&   MOF& 1.9e2&       & 6.3  &   ---\\
                        $\A_9 $& 5.9e2& 1.6  &  3.9e1&    ---&  ---&       $\Delta_{27}$&   MOF&   MOF&       & 1.3e1& 2.0  \\
                      $\A_{10}$& 2.8e4& 1.6e1&  8.3e2&  9.6  &  ---&        $\C_{20,20}$& 2.1e1& 1.6  &       &   ---&   ---\\
                      $\A_{11}$&   MOF& 2.4e2&    MOF&  2.0e2&4.7  &        $\C_{20,21}$& 4.8e1& 3.2  &  3.0e2&   ---&   ---\\
                         $\B_5$& 2.3e2& 8.8e0&  9.9e1&    ---&  ---&        $\C_{20,22}$& 2.6e2& 6.3  &    MOF&   ---&   ---\\
                         $\B_6$&   MOF&   MOF&    MOF&  2.7e3&3.1e2&        $\C_{20,23}$&   MOF& 1.2e1&    MOF&   ---&   ---\\
                 $\square_{10}$& 2.8  &   ---&    ---&    ---&  ---&        $\C_{20,24}$&   MOF& 2.6e1&    MOF& 1.6  &   ---\\
                 $\square_{11}$& 1.9e1&   ---&  1.9  &    ---&  ---&        $\C_{20,25}$&      & 4.4e1&       & 5.8  &   ---\\
                 $\square_{12}$& 1.2e2&   ---&  8.2  &    ---&  ---&        $\C_{20,26}$&      & 1.1e2&       & 2.4e1& 2.9  \\
                 $\square_{13}$& 8.9e2& 2.0  &  3.7e1&    ---&  ---&        $\C_{20,27}$&      & 2.7e2&       & 1.1e2& 9.0  \\
                 $\square_{14}$& 9.6e3& 9.0  &  1.9e2&  3.0  &  ---&        $\C_{20,28}$&      & 8.1e2&       & 4.3e2& 2.8e1\\
                 $\square_{15}$&   MOF& 3.2e1&    MOF&  1.8e1&  ---&        $\C_{20,29}$&      &   MOF&       & 1.8e3& 9.8e1\\
                 $\square_{16}$&   MOF& 1.5e2&    MOF&  1.1e2&3.3  &           $L(\A_3)$& 6.3e2& 1.7e1&  2.1e2&   ---&   ---\\
                 $\square_{17}$&   MOF& 9.4e2&    MOF&  6.9e2&2.0e1&      $L(\C_{4,8 })$& 1.1  &   ---&    ---&   ---&   ---\\
                 $\square_{18}$&   MOF&   MOF&    MOF&  4.4e3&1.3e2&      $L(\C_{4,9 })$& 5.3  &   ---&  1.4  &   ---&   ---\\
                   $\C_{10,20}$& 3.3e1&   ---&  2.2  &    ---&  ---&      $L(\C_{4,10})$& 2.3e1&   ---&  5.5  &   ---&   ---\\
                   $\C_{10,21}$& 6.5e1&   ---&  3.6  &    ---&  ---&      $L(\C_{4,11})$& 1.0e2& 2.0  &  2.1e1&   ---&   ---\\
                   $\C_{10,22}$& 1.3e2&   ---&  6.2  &    ---&  ---&      $L(\C_{4,12})$& 4.1e2& 6.9  &  7.3e1&   ---&   ---\\
                   $\C_{10,23}$& 2.6e2&   ---&  1.0e1&    ---&  ---&      $L(\C_{4,13})$& 1.7e3& 2.1e1&  2.5e2& 1.4  &   ---\\
                   $\C_{10,24}$& 5.9e2& 1.2  &  1.8e1&    ---&  ---&      $L(\C_{4,14})$&   MOF& 7.0e1&  8.0e2& 4.6  &   ---\\
                   $\C_{10,25}$& 1.3e3& 1.7  &  3.0e1&    ---&  ---&      $L(\C_{5,8 })$& 1.5  &   ---&    ---&   ---&   ---\\
                   $\C_{10,26}$& 2.6e3& 2.5  &  5.1e1&  1.6  &  ---&      $L(\C_{5,9 })$& 8.7  &   ---&  1.9  &   ---&   ---\\
                   $\C_{10,27}$& 5.0e3& 3.5  &  8.8e1&  2.7  &  ---&      $L(\C_{5,10})$& 5.0e1& 1.2  &  8.5  &   ---&   ---\\
                   $\C_{10,28}$& 9.1e3& 4.9  &  1.4e2&  4.5  &  ---&      $L(\C_{5,11})$& 2.6e2& 4.7  &  3.6e1&   ---&   ---\\
                   $\C_{10,29}$& 1.6e4& 7.0  &  2.3e2&  7.3  &  ---&      $L(\C_{5,12})$& 1.3e3& 1.6e1&  1.5e2&   ---&   ---\\
                   $\C_{10,30}$& 2.6e4& 9.9  &  3.6e2&  1.2e1&  ---&      $L(\C_{5,13})$&   MOF& 5.0e1&  5.5e2& 3.5  &   ---\\
                   $\C_{10,31}$& 4.2e4& 1.4e1&  5.7e2&  1.9e1&  ---&      $L(\C_{5,14})$&   MOF& 1.4e2&    MOF& 1.4e1& 2.0  \\
                   $\C_{10,32}$& 6.6e4& 2.0e1&    MOF&  2.9e1&  ---&  $L(\C_{4,6 }^\op)$& 2.2  &   ---&    ---&   ---&   ---\\
                   $\C_{10,33}$& 1.0e5& 2.9e1&       &  4.4e1&1.4  &  $L(\C_{4,7 }^\op)$&   MOF& 5.3e1&  5.1e2& 2.1  &   ---\\
                   $\C_{10,34}$& 1.5e5& 4.0e1&       &  6.6e1&2.2  &  $L(\C_{4,8 }^\op)$&   MOF&   MOF&    MOF& 1.3e3& 1.6e2\\
                   $\C_{10,35}$& 2.2e5& 5.7e1&       &  9.8e1&3.1  &  $L(\C_{5,7 }^\op)$& 1.6e2& 8.4  &  6.7e1&   ---&   ---\\
                   $\C_{10,36}$&      & 8.0e1&       &       &4.6  &  $L(\C_{5,8 }^\op)$&   MOF&   MOF&    MOF& 3.6e3& 4.3e2\\
                   $\C_{10,37}$&      & 1.1e2&       &       &6.8  &  $L(\square_5^\op)$& 1.2e1&   ---&  4.7  &   ---&   ---\\
                   $\C_{10,38}$&      & 1.5e2&       &       &1.0e1&  $L(\square_6^\op)$& 2.4e2& 9.7  &  9.4e1&   ---&   ---\\
                   $\C_{10,39}$&      & 2.1e2&       &       &1.6e1&  $L(\square_7^\op)$&   MOF& 1.2e2&  1.6e3& 4.6  &   ---\\
                   $\C_{10,30}$&      & 2.9e2&       &       &2.5e1&  $L(\square_8^\op)$&   MOF&   MOF&    MOF& 6.5e1& 1.0e1\\
                   $\C_{10,41}$&      & 3.8e2&       &       &4.0e1& $L(\square_4)$     &   MOF& 2.6e2&  3.5e3& 9.9  & 1.4  \\
                   $\C_{10,42}$&      &   MOF&       &       &6.7e1&           $\K_{12}$&      & 1.5  &       &      &   ---\\
                          $\Hi$&   MOF&      &    MOF&  4.5e2&  ---&           $\K_{13}$&      & 1.6e1&       &      &   ---\\
                         $\P_6$& 7.8  &   ---&  2.6  &    ---&  ---&           $\K_{14}$&      & 1.4e2&       &      & 1.5  \\
                         $\P_7$& 2.2e3& 1.7e1&  4.3e2&  6.0  &  ---&           $\K_{15}$&      & 2.6e3&       &      & 2.1e1\\
                  $\Delta_{16}$& 2.1  &   ---&    ---&    ---&  ---&           $\K_{16}$&      &   MOF&       &      & 2.4e2\\
                  $\Delta_{17}$& 4.6  &   ---&  2.0  &    ---&  ---\\

\end{tabular}
}
\end{center}

\bibliography{references}
\bibliographystyle{amsplain}

\end{document}